\documentclass[10pt]{amsart}
\usepackage{amsmath}
\usepackage{amsthm}
\usepackage{amssymb}
\usepackage{a4wide}
\usepackage{color}
\usepackage[pagewise,mathlines]{lineno}
\usepackage{hyperref}
\usepackage{amsfonts}

\usepackage{mathtools}
\mathtoolsset{showonlyrefs}

\newcommand*{\N}{\mathbb N}
\newcommand*{\R}{\mathbb R}

\newtheorem{lemma}{Lemma}[section]
\newtheorem{cor}{Corollary}[section]
\newtheorem{prop}{Proposition}[section]
\newtheorem{theorem}{Theorem}[section]

\newtheorem{remark}{Remark}[section]
\newtheorem{example}{Example}[section]

\begin{document}

\title[The spectrum set of a nonlocal Dirichlet problem]
{Remarks on the spectrum of  a nonlocal Dirichlet problem}

\author[R. D. Benguria and M. C. Pereira]{Rafael D. Benguria and Marcone C. Pereira}

\address{Rafael D. Benguria
\hfill\break\indent Dpto. de F\'isica, P. Universidad Cat\'olica de Chile,
\hfill\break\indent Casilla 306, Santiago 22, Chile. } 
\email{{\tt rbenguri@fis.puc.cl} \hfill\break\indent {\it
Web page: }{\tt www.fis.puc.cl/$\sim$rbenguri/}}

\address{Marcone C. Pereira
\hfill\break\indent Dpto. de Matem{\'a}tica Aplicada, IME,
Universidade de S\~ao Paulo, \hfill\break\indent Rua do Mat\~ao 1010, 
S\~ao Paulo - SP, Brazil. } \email{{\tt marcone@ime.usp.br} \hfill\break\indent {\it
Web page: }{\tt www.ime.usp.br/$\sim$marcone}}


\begin{abstract} 
In this paper we analyse the spectrum of nonlocal Dirichlet problems with non-singular kernels in bounded open sets.
The novelty is two fold. First we study the continuity of eigenvalues with respect to domain perturbation via Lebesgue measure. 
Next, under additional smooth conditions on the kernel and domain, we prove differentiability of simple eigenvalues computing their first derivative discussing extremum problems for eigenvalues. 
\end{abstract}

\maketitle

\section{Introduction}
\label{Sect.intro}
\setcounter{equation}{0}

In this note, we discuss the spectrum set of a nonlocal equation with non-singular kernels and Dirichlet conditions in bounded open sets $\Omega \subset \R^N$.
We consider the nonlocal eigenvalue problem
\begin{equation} \label{eigeneq}
\left\{
\begin{gathered}
(J* u)(x) - u(x) = - \lambda \, u(x), \quad x \in \Omega, \\
u(x) = 0, \quad x \in \R^N \setminus \Omega,
\end{gathered} 
\right.
\end{equation}
where $J*u$ stands for the usual convolution 
$$
(J * u)(x) = \int_{\R^N} J(x-y) u(y) dy
$$
with a kernel $J$. Through out this article the function $J$ satisfies the hypotheses 
$$
{\bf (H)} \qquad 
\begin{gathered}
J \in \mathcal{C}(\R^N,\R) \textrm{ is a nonnegative function, spherically symmetric and radially decreasing } \\ 
\textrm{ with } J(0)>0 \textrm{ and } 
\int_{\R^N} J(x) \, dx = 1.
\end{gathered}
$$

Our main goal is to study the continuity of the spectrum set with respect to the variation of the domain $\Omega$. 
Next, assuming $J$ and $\Omega$ are still $\mathcal{C}^1$-regular, we also show differentiability of simple eigenvalues computing an expression for their  first derivative allowing $\Omega$ to vary in the set of open sets which are $\mathcal{C}^1$-diffeomorphic.

Notice that analysing the spectral properties of \eqref{eigeneq} is equivalent to study the spectrum of 
the linear operator $\mathcal{B}_\Omega: W_\Omega \mapsto W_\Omega$ where
$W_\Omega = \{ u \in L^2(\R^N) \, : \, u(x) \equiv 0 \textrm{ in } \R^N \setminus \Omega \}$
and
\begin{equation} \label{Bdef}
\mathcal{B}_\Omega u(x) \equiv u(x) - \int_\Omega J(x-y) u(y) dy, \quad x \in \Omega. 
\end{equation}

Moreover, one has that the operator $\mathcal{B}_\Omega$ is the sum of the identity on the Hilbert space $W_\Omega$ minus the compact and self-adjoint operator 
$\mathcal{J}_{\Omega}: L^2(\Omega) \mapsto L^2(\Omega)$ which is given by the convolution  
\begin{equation} \label{opJ}
\mathcal{J}_\Omega u (x) = (J*u)(x), \quad x \in \Omega. 
\end{equation}

We will see that there exists a precise relationship between the spectrum of the operators $\mathcal{B}_\Omega$ and $\mathcal{J}_{\Omega}$. 
Indeed, the continuity properties for the eigenvalues of $\mathcal{B}_\Omega$ will be obtained by an accurate analysis of the spectrum of $\mathcal{J}_{\Omega}$ via perturbation theory for linear operators developed in \cite{Kato}.
Here, we mention our main result in this direction which guarantees the convergence of eigenvalues when the Lebesgue measure of symmetric difference of involved open sets vanishes.

\begin{theorem} \label{singleeig}
Let $\Omega_n \subset \R^N$ be a sequence of bounded open sets with 
$$
|\Omega \setminus \Omega_n| + |\Omega_n \setminus \Omega| \to 0, \quad \textrm{ as } n \to \infty
$$
for some bounded open set $\Omega \subset \R^N$.

Then, if $\mu(\Omega)$ is an eigenvalue of ${\mathcal{J}}_{\Omega}$, there exists a sequence of eigenvalues $\mu(\Omega_n) \in \sigma({\mathcal{J}}_{\Omega_n})$ such that 
$$
\mu(\Omega_n) \to \mu(\Omega), \quad \textrm{ as } n \to \infty.
$$

In particular, if $\lambda(\Omega)$ is an eigenvalue of ${\mathcal{B}}_{\Omega}$, there exists a sequence of eigenvalues $\lambda(\Omega_n) \in \sigma(\mathcal{B}_{\Omega_n})$ with
$$
\lambda(\Omega_n) \to \lambda(\Omega), \quad \textrm{ as }  n \to \infty.
$$
\end{theorem}

 Next, we follow the approach introduced in \cite{Henry} to perturb $\Omega$
 in order to take derivatives of simple eigenvalues with respect to the domain. 
 More precisely, if $\Omega \subset \R^N$ is a $\mathcal{C}^1$-regular open bounded set, and 
$h:\Omega \mapsto \R^N$ is a $\mathcal{C}^1$-diffeomorphism to its image, we define the composition map
$$
h^* v(x) = (v \circ h)(x), \quad x \in \Omega,
$$
for any $v$ set on $h(\Omega)$. $h^*: L^2(h(\Omega)) \mapsto L^2(\Omega)$ is an isomorphism with 
$(h^*)^{-1} = (h^{-1})^*$. 

For such imbedding $h$ and bounded region $\Omega$, one can introduce the nonlocal Dirichlet operator $\mathcal{B}_{h(\Omega)}$ on the perturbed open set $h(\Omega)$ by 
\begin{equation} \label{Bh}
\left( \mathcal{B}_{h(\Omega)} v \right)(y) =  v(y) - \int_{h(\Omega)} J(y - w) v(w) dw, \quad y \in h(\Omega),
\end{equation}
with $\mathcal{B}_{h(\Omega)}: W_{h(\Omega)} \mapsto W_{h(\Omega)}$.
On the other hand, we can use $h^*$ to set $h^* \mathcal{B}_{h(\Omega)} {h^*}^{-1}: W_\Omega \mapsto W_\Omega$ by
\begin{equation} \label{hBh}
h^* \mathcal{B}_{h(\Omega)} {h^*}^{-1} u(x) = \int_{h(\Omega)} J(h(x) - w) (u \circ h^{-1})(w) dw, 
\quad \forall x \in \Omega.
\end{equation}

It is known that expressions \eqref{Bh} and \eqref{hBh} are the customary manner to describe 
motion or deformation of regions. 
Form \eqref{Bh} is called the Lagrangian description, and \eqref{hBh} the  Eulerian one. 
The former is written in a fixed coordinate system while the Lagrangian does not. 
Also,  
\begin{equation} \label{eqBint}
h^* \mathcal{B}_{h(\Omega)} {h^*}^{-1} u(x) = v(y) - \int_{h(\Omega)} J(y - w) v(w) dw = \left( \mathcal{B}_{h(\Omega)} v \right)(y)
\end{equation}
if we take $y=h(x)$ and $v(y) = (u \circ h^{-1})(y) = {h^*}^{-1} u(y)$ for $y \in h(\Omega)$.

In this way, we perturb our eigenvalue problem \eqref{eigeneq}. We take imbeddings $h:\Omega \mapsto \R^N$ varying in the set of diffeomorphisms  $\mathop{\rm Diff}^1(\Omega)$ studying the eigenvalues of the operators \eqref{Bh} and \eqref{hBh} which are the same. 
We have the following result concerning the  derivative of simple eigenvalues.

\begin{theorem} \label{cderi}
Let $\lambda_0$ be a simple eigenvalue for $\mathcal{B}_{\Omega}$ with corresponding normalized eigenfuction 
$u_0$ and $J \in \mathcal{C}^1(\R^N, \R)$ satisfying $(\rm{H})$. 
Then, there exists a neighbourhood $\mathcal{V}$ of the inclusion $i_{\Omega} \in \mathop{\rm Diff}^1(\Omega)$, 
and $\mathcal{C}^1$-functions 
$(u_h, \lambda_h)$ from $\mathcal{V}$ into $L^2(\Omega) \times \mathbb{R}$ which satisfy 
$h^* \mathcal{B}_{h(\Omega)} {h^*}^{-1} u_h(x) = \lambda_h u(x)$, $x \in \Omega$,
with $u_h \in \mathcal{C}^1(\Omega)$.  
Also, $\lambda_h$ is a simple eigenvalue, $(\lambda_{i_{\Omega}}, u_{i_\Omega} ) = ( \lambda_0, u_0)$, and the domain derivative is given by 
\begin{equation} \label{dodeint}
\frac{\partial \lambda}{\partial h}(i_\Omega) \cdot V = - (1 - \lambda_0) \int_{\partial \Omega} u_0^2 \; V \cdot N_\Omega dS \quad 
\mbox{for all }  V \in \mathcal{C}^1(\Omega, \R^N)
\end{equation}
where $\partial \Omega$ denotes the boundary of $\Omega$ and $N_\Omega$ its normal vector.
\end{theorem}

At this point, it is worth noticing that we are improving here results from \cite{GR} where the domain perturbation to 
the first eigenvalue of \eqref{eigeneq} was considered and formula \eqref{dodeint} was first obtained.  
There, the authors have used the variational formulation of the first eigenvalue and the positivity of the corresponding 
eigenfunction which holds just in this particular case. Our result is more general since it holds for any simple 
eigenvalue also showing smooth persistence.

Finally, we mention some authors as \cite{ElLibro, Fife, HMMV} associate $J$ under conditions $(\rm{H})$ to a radial 
probability density calling equation \eqref{eigeneq} a nonlocal analogous to the Dirichlet boundary conditions problem 
to the Laplacian. Indeed, several continuous models for species and human mobility have been proposed using such 
nonlocal approach, in order to look for more realistic dispersion equations \cite{XF, BC, YY}.
Recall that hostile surroundings are modeled by the Dirichlet condition as in \eqref{eigeneq}.  
Besides the applied models with such kernels, the mathematical interest is mainly due to the fact that, in general, 
there is no regularizing effect and therefore no general compactness tools are available. 

The paper is organized as follows: in Section \ref{pre}, we show some preliminary results concerning to the spectrum 
of $\mathcal{J}_{\Omega}$ and $\mathcal{B}_{\Omega}$ also discussing isoperimetric inequalities for $\mathcal{B}
_\Omega$. Such inequalities are an analogue of Rayleigh-Faber-Krahn and Hong-Krahn-Szeg\"o inequalities and 
have been recently obtained for $\mathcal{J}_\Omega$ in \cite{RS}. For  a recent review on isoperimetric inequalities 
we refer to \cite{Beng}. 

In Section \ref{contS}, we study the continuity of eigenvalues with respect to $\Omega$. We also take into account 
recent results concerning the convergence of eigenvalues 
posed in oscillating and perforated domains. Finally, in  Section \ref{deriS}, we obtain the stability of a simple eigenvalue  
with respect to the  variation of smooth domains performed by imbeddings, proving Theorem \ref{cderi}.

\section{Basic facts and preliminary results} \label{pre}

Let us first discuss the operator $\mathcal{J}_{\Omega}: L^2(\Omega) \mapsto L^2(\Omega)$ given by 
the convolution \eqref{opJ}.
Notice $\mathcal{J}_{\Omega}$ is bounded, compact and self-adjoint satisfying 
$$
\| {\mathcal{J}}_\Omega \|_{L^2(\Omega)} \leq |\Omega| \| J \|_{L^\infty(\R^N)}.
$$ 
Such a proof is straightforward and can be found for instance in \cite{AS, CHvol1}.
In the sequel, we mention other properties with respect to its spectral set which are also consequence of classical results from functional analysis.

\begin{remark} \label{EiPro}
Since $\mathcal{J}_{\Omega}$ is compact and self-adjoint, one may obtain, for instance from \cite[Theorem 2.10 Chapter V]{Kato}, that the spectrum $\sigma(\mathcal{J}_{\Omega})$ consists of at most a countable number of real eigenvalues with finite multiplicities, possible excepting zero.
Let us enumerate their eigenvalues in decreasing order of magnitude 
$$
|\mu_1| \geq |\mu_2| \geq ...
$$
If $P_1$, $P_2$, ... are the associated eigenprojections of $\mathcal{J}_{\Omega}$, then $P_i$ are orthogonal and 
self-adjoint with finite dimensional range. Also, we have the spectral representation 
$$
\mathcal{J}_{\Omega} = \sum_{i\geq1} \mu_i P_i
$$ 
in the sense of convergence in norm with projections forming a complete orthogonal family together with the 
orthogonal projection $P_0$ on the null space of $\mathcal{J}_{\Omega}$. 
\end{remark}

\begin{remark} \label{EiPro2}
From \cite[Theorem 2.10 Chapter V]{Kato}, we have that $0 \in \sigma(\mathcal{J}_{\Omega})$. Also, if there exists 
an infinite sequence of distinct eigenvalues $\mu_i$, then $\mu_i \to 0$ as $i \to +\infty$, and then, zero belongs to 
the essential spectrum $\sigma_{ess}(\mathcal{J}_{\Omega})$. On the other hand, if the set of eigenvalues is finite, 
its null space is not trivial, indeed, it is an infinite dimensional subspace of $L^2(\Omega)$. 
\end{remark}

\begin{remark} \label{mu1}
We note that $|\mu_1|$ is equal to the spectral radius of $\mathcal{J}_{\Omega}$ which coincides with its norm
$$
|\mu_1| = \lim_{n\to +\infty} \| \mathcal{J}_{\Omega}^n \|^{1/n} = \| \mathcal{J}_{\Omega} \|.
$$
Moreover, it is known from \cite{AS, RS}, that the first eigenvalue $\mu_1$ is positive, simple, whose corresponding 
eigenfunction $u_1$ can be chosen strictly positive in $\Omega$. 
\end{remark}

Since the eigenvalues $\mu_i$ have finite multiplicity, we can set them in a decreasing order of magnitude also 
taking account their multiplicity. Hence, we denote by  $u_1$, $u_2$, ... the corresponding eigenfunctions for each 
eigenvalue $\mu_i$ setting 
$$
\mathcal{J}_{\Omega} u_i(x) = \mu_i(\Omega) u_i(x). 
$$

Now, let us denote the range of $\mathcal{J}_{\Omega}$ by ${\rm R}(\mathcal{J}_\Omega)$. 
Since $\mathcal{J}_{\Omega}$ is self-adjoint, ${\rm R}(\mathcal{J}_\Omega)$ is orthogonal to 
the kernel  of  $\mathcal{J}_\Omega$, ${\rm ker}(\mathcal{J}_\Omega)$, setting a useful  decomposition for $L^2(\Omega)$. 
From Remark \ref{EiPro}, one gets 
$$
L^2(\Omega) = {\rm R}(\mathcal{J}_\Omega) \oplus {\rm ker}(\mathcal{J}_\Omega).
$$  
We still have the following result concerning  ${\rm R}(\mathcal{J}_\Omega)$. 

\begin{lemma} \label{lemI}
Assume ${\rm R}(\mathcal{J}_\Omega)$ is finite dimensional. 

Then, there exist a set of normalized eigenfunctions $\{ u_1, ..., u_m  \} \subset L^2(\Omega)$, 
associated to nonzero eigenvalues $\mu_i(\Omega)$, such that 
\begin{equation} \label{eq210}
J(x-y) = \sum_{i=1}^m \mu_i(\Omega) u_i(x) u_i(y), \quad \textrm{ a.e. } \Omega.
\end{equation}

In particular,
$J(x) = \sum_{i=1}^m \mu_i(\Omega) u_i(x) u_i(0)$ a.e. $\Omega$, 
and $J(0) |\Omega| = \sum_{i=1}^m \mu_i(\Omega)$.
\end{lemma}
\begin{proof}
First, we recall that $L^2(\Omega)$ is the direct sum of ${\rm R}(\mathcal{J}_\Omega)$ and ${\rm ker}(\mathcal{J}
_\Omega)$. 
Thus, if ${\rm R}(\mathcal{J}_\Omega)$ is finite dimensional, by \ref{EiPro} again, there exist $\{ u_1, ..., u_m  \} 
\subset L^2(\Omega)$ given by orthogonal and normalized eigenfunctions of $\mathcal{J}_\Omega$, associated to 
nonzero eigenvalues $\mu_i(\Omega)$ such that 
$$
{\rm R}(\mathcal{J}_\Omega) = [u_1, ..., u_m].
$$
 
Hence, we can take the orthogonal projections $P_i$ as 
$$
P_i u(x) = \mu_i(\Omega) \left(  \int_\Omega u_i(y) u(y) \right) u_i(x), \quad x \in \Omega.
$$
For all $u \in L^2(\Omega)$, we have 
$$
\mathcal{J}_\Omega u(x) = \int_\Omega J(x-y) u(y) dy = \sum_{i=1}^m \mu_i(\Omega) \left( \int_\Omega u_i(y) u(y) dy \right) u_i(x), \quad x \in \Omega.
$$
Consequently,
$$
0 = \int_\Omega \left( J(x-y) - \sum_{i=1}^m \mu_i(\Omega) u_i(x) u_i(y) \right) u(y) dy, \quad \forall u \in L^2(\Omega) 
\textrm{ and } \forall x \in \Omega,
$$
completing the proof.
\end{proof}

Now, let us consider the operator $\mathcal{B}_\Omega: W_\Omega \mapsto W_\Omega$ defined by \eqref{Bdef}.
Since $\mathcal{B}_\Omega$ is a scalar combination of the identity and the self-adjoint operator $\mathcal{J}
_{\Omega}$, $\mathcal{B}_{\Omega}$ is also a bounded self-adjoint operator in $L^2(\Omega)$.

\begin{remark} \label{propeig}
We notice that:
\begin{itemize}
\item[a)] $\lambda(\Omega) \in \sigma(\mathcal{B}_\Omega)$ is an eigenvalue, if and only if, there exists $u \in 
L^2(\Omega)$, $u \neq 0$, with $u(x) \equiv 0$ in $\R^N \setminus \Omega$, satisfying equation \eqref{eigeneq} for 
this same $\lambda(\Omega)$. 

\item[b)] $u \in L^2(\Omega)$ is a fixed point of $\mathcal{B}_\Omega$, if and only if, $u$ belongs to the null set of $
\mathcal{J}_{\Omega}$.

\item[c)] $\lambda(\Omega) \in \sigma(\mathcal{B}_\Omega)$ is an eigenvalue, if and only if, $1 - \lambda(\Omega)$ 
is an eigenvalue of the compact operator $\mathcal{J}_{\Omega}$. Also, 
$0 \in \sigma_{ess}(\mathcal{J}_{\Omega})$, if 
and only if, $1 \in \sigma_{ess}(\mathcal{B}_\Omega)$, and zero is an eigenvalue of $\mathcal{J}_{\Omega}$, if and 
only if, $1 \in \sigma(\mathcal{J}_{\Omega})$. 
\item[d)] From Remark \ref{mu1}, we know that the first eigenvalue of $\mathcal{B}_\Omega$, which is given by $
\lambda_1(\Omega) = 1 - \mu_1(\Omega)$, it is associated to a strictly positive eigenfunction which is also simple 
with 
$$
\lambda_1(\Omega) = 1 - \| \mathcal{J}_{\Omega} \| < 1.
$$
\item[e)] Further, since we are assuming $\int_{\R^N} J(y) dy =1$, we have 
$$
\frac{1}{2} \int_{\R^N} \int_{\R^N} J(x-y) (u(y) - u(x))^2 dy dx = \| u \|_{L^2(\R^N)}^2 - \int_{\R^N}\int_{\R^N}J(x-y) u(y) 
u(x)\, dy \, dx,
$$ 
and then, we get from  $\rm{d})$ that 
\begin{equation} \label{eq350}
\lambda_1(\Omega) = \inf_{u \neq 0 \textrm{ in } W_\Omega} \frac{\frac{1}{2} \int_{\R^N} \int_{\R^N} J(x-y) (u(y) - 
u(x))^2 dy dx}{\| u \|_{L^2(\R^N)}^2}.
\end{equation}
For more details, see \cite{ElLibro, GR}.
\end{itemize}
\end{remark}

Let us take $u_1$, the first positive eigenfunction of $\mathcal{B}_\Omega$. It follows from \eqref{eigeneq} that 
\begin{eqnarray*}
- \lambda_1(\Omega) \int_\Omega (u_1(x))^2 dx & = & \int_\Omega u_1(x) \int_\Omega J(x-y) ( u_1(y) - u_1(x) ) dy \\
& = & - \frac{1}{2} \int_\Omega \int_\Omega J(x-y) ( u_1(y) - u_1(x) )^2 dy dx \leq 0.
\end{eqnarray*}
Thus, $0 \leq \lambda_1(\Omega) < 1$ with $\lambda_1(\Omega) =0$, if and only if, $u_1$ is a positive constant.
Now, due to \cite[Proposition 2.2]{ElLibro}, one can get that $J*u(x) - u(x) \equiv 0$ in $\Omega$ with $u(x) \equiv 0$ in $\R^N \setminus \Omega$, if and only if, $u(x) \equiv 0$ in $\R^N$. Hence, we conclude that
\begin{equation} \label{300}
0 < \lambda_1(\Omega) < 1 \quad \textrm{ and } \quad 0 < \| \mathcal{J}_{\Omega} \| < 1
\end{equation}  
for any bounded open set $\Omega$.

Consequently, we obtain from \eqref{300} that $\mathcal{B}_\Omega$ is a perturbation of the identity 
being an invertible operator with continuous inverse given by
$
\mathcal{B}_\Omega^{-1} u = (I - \mathcal{J}_{\Omega})^{-1} u = \sum_{n=0}^{\infty} \mathcal{J}_{\Omega}^n u.
$

\begin{remark}
Others informations and properties concerning the operators $\mathcal{J}_{\Omega}$ and $\mathcal{B}_{\Omega}$, and their spectrum set, can be seen for instance in \cite{Vivian, AS, CHvol1} and references therein. Moreover, it is important to know that all  the results discussed to this point remain valid substituting the radial condition on the function $J$ with the even one, i.e., assuming $J(-x)=J(x)$.
\end{remark}

Finally, let us just mention some isoperimetric inequalities for the  
first and second eigenvalues of  $\mathcal{B}_\Omega$.
Due to the symmetric condition imposed on the kernel $J$, an analogue of Rayleigh-Faber-Krahn and Hong-Krahn-Szeg\"o inequalities for $\mathcal{J}_\Omega$ have been shown in  \cite{RS}. Hence, since Remark \ref{propeig} gives a precise relationship between the spectrum of $\mathcal{J}_\Omega$ and $\mathcal{B}_\Omega$, we can easily extend the results from \cite{RS} to the Dirichlet problem \eqref{eigeneq}.

Concerning the Rayleigh-Faber-Krahn inequality, we have the following result:

\begin{cor} \label{corl1}
Let $\Omega^*$ denote an open ball with same measure as $\Omega$.
Then, under conditions $(\rm{H})$, the ball $\Omega^*$ is a minimizer for the first eigenvalue of $\mathcal{B}_\Omega$, i.e., 
$$
\lambda_1(\Omega) \geq \lambda_1(\Omega^*).
$$
\end{cor}
\begin{proof}
It has been seen at \cite[Theorem 2.1]{RS} that the first eigenvalue $\mu_1(\Omega)$ of $\mathcal{J}_\Omega$ achieves its maximum among open sets of given volume at the ball $\Omega^*$. That is, $\mu_1(\Omega) \leq \mu_1(\Omega^*)$. Hence, we get the result from expression $\lambda_1(\Omega) = 1 - \mu_1(\Omega)$ given by Remark \ref{propeig}.
\end{proof}

 In Section \ref{deriS}, we give an example which shows that the first eigenvalue of \eqref{eigeneq} does not possess a maximizer among open bounded sets even with a fixed measure. 
Now we consider the minimizer of the second eigenvalue of $\mathcal{B}_\Omega$ among open sets of given 
volume. As we are going to see, the minimizer is no longer one ball, but the union of two identical balls whose mutual 
distance is going to infinity. It is an analogue of the Hong-Krahn-Szeg\"o inequality \cite{Henrot} and it has 
been proven in  \cite[Theorem 2.3]{RS} for the compact operator $\mathcal{J}_\Omega$. 
%
First, we prove the existence of $\lambda_2(\Omega)$ (and $\mu_2(\Omega)$) for any $\Omega \subset \R^N$. 

\begin{prop} \label{eigen2}
Under conditions $(\rm{H})$, we have ${\rm dim}({\rm R}(\mathcal{J}_\Omega)) \geq 2$. In particular, 
there exists $\lambda_2(\Omega)$ for any bounded open domain $\Omega \subset \R^N$.
\end{prop}

\begin{proof}
Let us suppose that $\mathcal{J}_\Omega$ is a one dimensional linear space. Then, by Lemma \ref{lemI}, taking $x=y$ in \eqref{eq210}, we have that 
$J(0) = \mu_1(\Omega) (u_1(x))^2$ in $\Omega$ where $\mu_1(\Omega)$ is the first eigenvalue of $\mathcal{J}_\Omega$ with corresponding normalized eigenfunction $u_1 \in L^2(\Omega)$.
Hence, we conclude that $u_1$ is a strictly positive constant which is a contradiction, since it satisfies \eqref{eq350} with 
$\lambda_1(\Omega) = 1 - \mu_1(\Omega)>0$. Finally, as $\lambda_1(\Omega)$ is a simple eigenvalue, it follows that there exists at least another larger eigenvalue of $\mathcal{B}_\Omega$.
\end{proof}

Now, let us optimize the second eigenvalue. 
\begin{cor}
Under hypothesis $(\rm{H})$, the minimum of the second eigenvalue of \eqref{eigeneq} among all bounded open sets with given volume is achieved by the disjoint union of two identical balls with mutual distance attaching to infinity. 
\end{cor}
\begin{proof}
The result is a direct consequence of the expression $\lambda_2(\Omega) = 1 - \mu_2(\Omega)$ and \cite[Theorem 2.3]{RS} where it has been proved that the maximum of $\mu_2(\Omega)$ is achieved in a disjoint union of identical balls with mutual distance going to infinity. 
\end{proof}

\section{Continuity of eigenvalues} \label{contS}

In this section we discuss the continuity of the eigenvalues with respect to $\Omega \subset \R^N$. Notice that 
this is not a trivial task since any change of $\Omega$ causes a change on the operator domain. In order to 
overcome this problem, we extend $\mathcal{J}_\Omega$ into a $L^2(D)$ for a larger bounded set $D \subset 
\R^N$.

Let us take $\Omega \subset D$. We define $\tilde{\mathcal{J}}_\Omega: L^2(D) \mapsto L^2(D)$ by 
$$
\tilde{\mathcal{J}}_\Omega u(x) = 
\left\{
\begin{array}{ll}
\mathcal{J}_\Omega u(x) & x \in \Omega \\
0 & x \in D \setminus \Omega
\end{array}
\right. .
$$
Notice that $\tilde{\mathcal{J}}_\Omega u(x) = {\mathcal{J}}_\Omega u(x)$ for all $x \in \Omega$, and then,  $
\tilde{\mathcal{J}}_\Omega$ is
an extension of $\mathcal{J}_\Omega$ into $L^2(D)$.
It is not difficult to see that $\tilde{\mathcal{J}}_\Omega$ is a compact and self-adjoint operator  acting on $L^2(D)$ 
with 
$$
\| \tilde{\mathcal{J}}_\Omega \|_{L^2(D)} \leq |\Omega| \| J \|_{L^\infty(\R^N)}.
$$ 

Thus, we can argue as in Remark \ref{EiPro} getting from \cite[Theorem 2.10 Chapter V]{Kato} that 
$\sigma(\tilde{\mathcal{J}}_{\Omega})$ consists of at most a countable number of real eigenvalues with finite 
multiplicities, possibly excepting zero.
We also enumerate their eigenvalues in decreasing order of magnitude 
$$
|\tilde{\mu}_1| \geq |\tilde{\mu}_2| \geq ...
$$
If $\tilde{P}_1$, $\tilde{P}_2$, ... are the associated eigenprojections, then $\tilde{P}_i$ are orthogonal and self-adjoint 
with finite dimensional range. 
Finally, we also get a spectral representation 
$$
\tilde{\mathcal{J}}_{\Omega} = \sum_{i\geq1} \tilde{\mu}_i \tilde{P}_i
$$ 
in the sense of convergence in norm with projections forming a complete orthogonal family together with the 
orthogonal projection $\tilde{P}_0$ on the null space of $\tilde{\mathcal{J}}_{\Omega}$. 

In the sequel, we first get conditions, in order to guarantee the continuity of the operators $\tilde{\mathcal{J}}
_\Omega$ with respect to $\Omega$. 
Next, we notice that the nonzero eigenvalues of $\tilde{\mathcal{J}}_\Omega$ and ${\mathcal{J}}_\Omega$ are 
equal. Here we study continuity via abstract results concerning  perturbations for linear operators dealt in \cite{Kato}.

\begin{lemma} \label{lemJt}
Let $\Omega_1$, $\Omega_2$ be two bounded open sets in $\R^N$. 
Then, there exists $C>0$ depending only on the measure of $\Omega_1$ and $\Omega_2$ such that 
$$
\| \tilde{\mathcal{J}}_{\Omega_1} - \tilde{\mathcal{J}}_{\Omega_2} \|_{D} \leq C \| J \|_{L^\infty(\R^N)} \left[  |
\Omega_1 \setminus \Omega_2| + |\Omega_2 \setminus \Omega_1| \right]^{1/2}.
$$
In particular,  
$$
\| \tilde{\mathcal{J}}_{\Omega_1} - \tilde{\mathcal{J}}_{\Omega_2} \|_{L^2(D)} \to 0
$$  
if $|\Omega_1 \setminus \Omega_2| + |\Omega_2 \setminus \Omega_1| \to 0$ and $\| J \|_{L^\infty(\R^N)} < \infty$.
\end{lemma}
\begin{proof}
Notice that
$$
\tilde{\mathcal{J}}_{\Omega_1}u(x) - \tilde{\mathcal{J}}_{\Omega_2}u(x) = 
\left\{
\begin{array}{cc}
\int_{\Omega_1} J(x-y) u(y) dy - \int_{\Omega_2} J(x-y) u(y) dy, & x \in \Omega_1 \cap \Omega_2, \\
\int_{\Omega_1} J(x-y) u(y) dy, & x \in \Omega_1 \setminus \Omega_2, \\
- \int_{\Omega_2} J(x-y) u(y) dy, & x \in \Omega_2 \setminus \Omega_1.
\end{array}
\right.
$$
Hence, if $x \in \Omega_1 \cap \Omega_2$, we get 
\begin{eqnarray*}
| \tilde{\mathcal{J}}_{\Omega_1}u(x) - \tilde{\mathcal{J}}_{\Omega_2}u(x) | & \leq & 
\| J \|_{L^\infty(\R^N)} \left[  \int_{\Omega_1 \setminus \Omega_2}  J(x-y) |u(y)| dy + \int_{\Omega_2 \setminus \Omega_1}  J(x-y) |u(y)| dy \right] \\
& \leq & \| J \|_{L^\infty(\R^N)} \| u \|_{L^2(D)} \left[  |\Omega_1 \setminus \Omega_2|^{1/2} + |\Omega_2 \setminus \Omega_1|^{1/2} \right].
\end{eqnarray*}

On the other hand, if $x \in (\Omega_1 \setminus \Omega_2) \cup (\Omega_2 \setminus \Omega_1)$, 
$$
| \tilde{\mathcal{J}}_{\Omega_1}u(x) - \tilde{\mathcal{J}}_{\Omega_2}u(x) |  \leq 
\| J \|_{L^\infty(\R^N)} \| u \|_{L^2(D)} \left( |\Omega_1|^{1/2} + |\Omega_2|^{1/2}  \right).
$$

Consequently, 
\begin{eqnarray*}
\int_{D} | \tilde{\mathcal{J}}_{\Omega_1}u(x) - \tilde{\mathcal{J}}_{\Omega_2}u(x) |^2 dx  & \leq & 
\int_{\Omega_1 \cup \Omega_2} | \tilde{\mathcal{J}}_{\Omega_1}u(x) - \tilde{\mathcal{J}}_{\Omega_2}u(x) |^2 dx \\
& \leq & \| J \|_{L^\infty}^2 \| u \|_{L^2(D)}^2 \left(  |\Omega_1 \setminus \Omega_2|^{1/2}  
+ |\Omega_2 \setminus \Omega_1|^{1/2}  \right)^2 |\Omega_1 \cap \Omega_2| \\
& & + \| J \|_{L^\infty}^2 \| u \|_{L^2(D)}^2 \left( |\Omega_1|^{1/2} + |\Omega_2|^{1/2}  \right)^2 \left( |\Omega_1 
\setminus \Omega_2| + |\Omega_2 \setminus \Omega_1| \right) \\
& \leq & 2 \| J \|_{L^\infty}^2 \| u \|_{L^2(D)}^2 \left( |\Omega_1 \setminus \Omega_2| 
+ |\Omega_2 \setminus \Omega_1| \right) \left(  |\Omega_1 \cup \Omega_2| + |\Omega_1| + |\Omega_2|  \right) 
\end{eqnarray*}
proving the result. 
\end{proof}

Next, let us see that the sets of nonzero eigenvalues of 
$\tilde{\mathcal{J}}_{\Omega}$ and ${\mathcal{J}}_{\Omega}$ are equal. 
\begin{lemma} \label{lemig}
A nonzero value $\mu$ is an eigenvalue of the operator $\tilde{\mathcal{J}}_{\Omega}$, if and only if, it is a nonzero 
eigenvalue for ${\mathcal{J}}_{\Omega}$. Furthermore, we have that their multiplicity is preserved. 
\end{lemma}
\begin{proof}
We have that  $\mu \neq 0$ is an eigenvalue of $\tilde{\mathcal{J}}_{\Omega}$, if and only if, there exists $u \neq 0$ 
in $L^2(D)$ with
$$
\tilde{\mathcal{J}}_{\Omega} u(x) = \mu u(x), \quad x \in D \subset \R^N.
$$
Thus, from definition of $\tilde{\mathcal{J}}_{\Omega}$, we get 
$$
{\mathcal{J}}_{\Omega} u(x) = \mu u(x), \quad x \in \Omega, 
$$
with $u(x) \equiv 0$ in $D \setminus \Omega$ since $\mu \neq 0$.
Consequently, $\mu$ is also  an eigenvalue of ${\mathcal{J}}_{\Omega}$ with corresponding eigenfunction $u$.
On the other hand,  if $\mu \neq 0$ is an eigenvalue of ${\mathcal{J}}_{\Omega}$ with corresponding 
nonzero $u \in L^2(\Omega)$, we have that the extension by zero of $u$ into $L^2(D)$ is also an 
eigenfunction of $\tilde{\mathcal{J}}_{\Omega}$ associated to $\mu$, completing the proof.
\end{proof}

Now, let $s_T = \{ \lambda_{p_1}, ..., \lambda_{p_k} \}$ be a collection of finite eigenvalues of a compact and 
self-adjoint operator $T$ and $P_{p_1}$, ..., $P_{p_k}$ their associated orthogonal eigenprojections.  We say that 
$s_T$ is a finite system of eigenvalues with multiplicity $m \in \N$, if the range ${\rm R}(P_{p_i})$ of $P_{p_i}$ is 
finite and satisfies 
$$\sum_{i=1}^{k} {\rm dim}({\rm R}(P_{p_i})) = m.$$ 
Notice we can associate to $s_T$ an orthogonal projection $P_{s_T}$ given by $P_{s_T} = \sum_i P_{p_i}$.
If in addition, all eigenvalues of $s_T$ are simple, we call $s_T$ a finite system of simple eigenvalues.

Our next result shows the persistence of a finite system of eigenvalues for $\tilde{\mathcal{J}}_{\Omega}$ when we 
perturb $\Omega$. As we shall see, this is a direct consequence of the continuity of the operators with respect to 
$\Omega$ in norm and abstract results from perturbation theory of linear operators shown in \cite{Kato}.


\begin{lemma} \label{theofse}
Let $s_{\tilde{\mathcal{J}}_{\Omega}} \subset \sigma(\tilde{\mathcal{J}}_{\Omega})$ be a finite system of eigenvalues with multiplicity $m \in \N$ and $\mathcal{V} \subset \R$ a neighborhood of $s_{\tilde{\mathcal{J}}_{\Omega}}$.
Then, for all $\varepsilon > 0$, there exist $\delta>0$ and a neighborhood $\mathcal{V}_\varepsilon \subset \mathcal{V}$ of $s_{\tilde{\mathcal{J}}_{\Omega}}$ depending on $s_{\tilde{\mathcal{J}}_{\Omega}}$, $\mathcal{V}$ and $\tilde{\mathcal{J}}_{\Omega}$, such that, if $\tilde \Omega \subset D \subset \R^N$ satisfies 
\begin{equation} \label{eqdelta}
|\Omega \setminus \tilde \Omega| + |\tilde \Omega \setminus \Omega| < \delta 
\end{equation}
then, $\tilde{\mathcal{J}}_{\tilde \Omega}$ also has a finite system of eigenvalues $s_{\tilde{\mathcal{J}}_{\tilde \Omega}}$ with multiplicity $m$ and $ s_{\tilde{\mathcal{J}}_{\tilde \Omega}} \subset \mathcal{V}_\varepsilon$.
Furthermore, the orthogonal projections $P_{s_{\tilde{\mathcal{J}}_{\Omega}}}$ and $P_{s_{\tilde{\mathcal{J}}_{\tilde \Omega}}}$ associated to the finite systems  $s_{\tilde{\mathcal{J}}_{ \Omega}}$ and $s_{\tilde{\mathcal{J}}_{\tilde \Omega}}$ satisfy 
$\| P_{s_{\tilde{\mathcal{J}}_{\Omega}}} - P_{s_{\tilde{\mathcal{J}}_{\tilde \Omega}}} \|_{L^2(D)} < \varepsilon$.
\end{lemma}
\begin{proof}
Since $s_{\tilde{\mathcal{J}}_{\Omega}}$ is a finite collection of eigenvalues and $\mathcal{V}$ is a given neighborhood, we can construct a finite collection of disjoint open disks $B_i$ in $\mathbb{C}$ with radius $r_i>0$ such that $s_{\tilde{\mathcal{J}}_{\Omega}} \subset (\cup_i \bar B_i) \cap \R \subset \mathcal{V}$ and $B_i \cap s_{\tilde{\mathcal{J}}_{\Omega}} = \tilde{\mu}_i(\Omega)$ for some eigenvalue $\tilde{\mu}_i(\Omega)$ of $\tilde{\mathcal{J}}_{\Omega}$.
For each $i$, let us consider the circle $\Gamma_i$ given by the boundary  $\partial B_i$ of $B_i$. Hence, for each $i$, we can separate $\sigma(\tilde{\mathcal{J}}_{\Omega})$ in two natural parts $\sigma_{i,1}(\tilde{\mathcal{J}}_{\Omega})$ and $\sigma_{i,2}(\tilde{\mathcal{J}}_{\Omega})$ where 
$\sigma_{i,1}(\tilde{\mathcal{J}}_{\Omega}) = \sigma(\tilde{\mathcal{J}}_{\Omega}) \cap B_i$
and 
$\sigma_{i,2}(\tilde{\mathcal{J}}_{\Omega}) = \sigma(\tilde{\mathcal{J}}_{\Omega}) \cap \bar{B}_i^c$,
and $L^2(\Omega) = M_{1,i} \oplus M_{i,2}$ where $M_{1,i}$ is the range of the orthogonal projection associated to $\tilde{\mu}_i(\Omega) \in B_i$, and $M_{2,i}$ is the enumerate union of all ranges given by the others eigenprojections and kernel of $\tilde{\mathcal{J}}_{\Omega}$.

It follows from Lemma \ref{lemJt}, \cite[Theorem 2.23, page 206]{Kato} and \cite[Theorem 3.16, page 212]{Kato} that, for all $\varepsilon>0$, there exist $\delta_i$ and $r_i>0$ depending just on $\tilde{\mathcal{J}}_{\Omega}$ and $\Gamma_i$ such that, if $\tilde \Omega$ satisfies \eqref{eqdelta}, then $\sigma(\tilde{\mathcal{J}}_{\tilde \Omega} )$ can be likewise separated by $\Gamma_i$ in two parts $\sigma_{i,1}(\tilde{\mathcal{J}}_{\tilde \Omega})$ and $\sigma_{i,2}(\tilde{\mathcal{J}}_{\tilde \Omega})$ with associated decomposition $L^2(\Omega) = \tilde M_{1,i} \oplus \tilde M_{i,2}$. 
$\tilde M_{1,i}$ and $\tilde M_{i,2}$ are respectively isomorphic with $M_{1,i}$ and $M_{i,2}$ and corresponding orthogonal projections $\varepsilon$-closed in operator norm. 
In particular, ${\rm dim}(\tilde M_{1,i}) = {\rm dim}(M_{1,i})$ and ${\rm dim}(\tilde M_{2,i}) = {\rm dim}(M_{2,i})$ and both $\sigma_{i,1}(\tilde{\mathcal{J}}_{\tilde \Omega})$ and $\sigma_{i,2}(\tilde{\mathcal{J}}_{\tilde \Omega})$ are nonempty if this is true for $\tilde{\mathcal{J}}_{\Omega}$.
Since we are considering a finite collection of eigenvalues, the result follows taking $\delta = \min_i \{ \delta_i \}$ and $\mathcal{V}_\varepsilon = \left( \cup_i B_i \right) \cap \mathcal{V}$.
\end{proof}

As a direct consequence of Remark \ref{propeig} and Lemmas \ref{lemig} and \ref{theofse}, we obtain the continuity of a finite system of eigenvalues for the  operators ${\mathcal{J}}_{\Omega}$ and ${\mathcal{B}}_{\Omega}$. We have the following result.
\begin{theorem} \label{theisy}
Let $s_{{\mathcal{J}}_{\Omega}} \subset \sigma({\mathcal{J}}_{\Omega})$ be a finite system of eigenvalues with multiplicity $m \in \N$ and $\mathcal{V} \subset \R$ a neighborhood of $s_{{\mathcal{J}}_{\Omega}}$.

Then, for all $\varepsilon > 0$, there exist $\delta>0$ and a neighborhood $\mathcal{V}_\varepsilon \subset \mathcal{V}$ of $s_{{\mathcal{J}}_{\Omega}}$ depending on $s_{{\mathcal{J}}_{\Omega}}$, $\mathcal{V}$ and ${\mathcal{J}}_{\Omega}$ such that, if $\tilde \Omega \subset D \subset \R^N$ satisfies 
\begin{equation} \label{eqdelta}
|\Omega \setminus \tilde \Omega| + |\tilde \Omega \setminus \Omega| < \delta 
\end{equation}
then, ${\mathcal{J}}_{\tilde \Omega}$ also has a finite system of eigenvalues $s_{{\mathcal{J}}_{\tilde \Omega}}$ with multiplicity $m$ and $ s_{{\mathcal{J}}_{\tilde \Omega}} \subset \mathcal{V}_\varepsilon$.

Furthermore, if $s_{{\mathcal{B}}_{\Omega}}$ is also a finite system of eigenvalues with multiplicity $m \in \N$ for the operator ${\mathcal{B}}_{\Omega}$, we have, under the  same condition \eqref{eqdelta}, the existence of a finite system of eigenvalues $s_{{\mathcal{B}}_{\tilde \Omega}} \subset \mathcal{V}_\varepsilon$ with multiplicity $m$.
\end{theorem}

We also notice the persistence of a finite system of simple eigenvalues. 
\begin{cor} \label{corsimple}
Let $s_{{\mathcal{J}}_{\Omega}} = \{ \mu_1(\Omega), ..., \mu_{k}(\Omega) \} \subset \sigma({\mathcal{J}}_{\Omega})$ be a finite system of simple eigenvalues with $s_{{\mathcal{J}}_{\Omega}} \subset \mathcal{V}$ for some open set $\mathcal{V} \subset \R$.

Then, for all $\varepsilon > 0$, there exist $\delta>0$ and a neighborhood $\mathcal{V}_\varepsilon \subset \mathcal{V}$ of $s_{{\mathcal{J}}_{\Omega}}$ depending on $s_{{\mathcal{J}}_{\Omega}}$, $\mathcal{V}$ and ${\mathcal{J}}_{\Omega}$ such that, if $\tilde \Omega \subset D \subset \R^N$ satisfies \eqref{eqdelta}, 
the operator ${\mathcal{J}}_{\tilde \Omega}$ also possesses a finite system of simple eigenvalue $s_{{\mathcal{J}}_{\tilde \Omega}} = \{  \mu_1(\tilde \Omega), ...,  \mu_{k}(\tilde \Omega) \} \subset \mathcal{V}_\varepsilon$.

Respectively, if $s_{{\mathcal{B}}_{\Omega}} = \{ \lambda_1(\Omega), ..., \lambda_{k}(\Omega) \} \subset \mathcal{V}$ 
is a finite system of simple eigenvalues for ${\mathcal{B}}_{\Omega}$, then there exists a finite system of simple eigenvalues $s_{{\mathcal{B}}_{\tilde \Omega}} = \{  \lambda_1(\tilde \Omega), ...,  \lambda_{k}(\tilde \Omega) \} \subset \mathcal{V}_\varepsilon$.
\end{cor}
\begin{proof}
Let us apply Lemma \ref{theofse} to each single system $\{ \mu_i(\Omega) \} \subset s_{{\mathcal{J}}_{\Omega}}$. Since $\mu_i(\Omega)$ is simple, for each $i=1, 2, ..., k$, there exists $\delta_i>0$ such that $\{ \mu_i(\tilde \Omega) \} \subset \sigma({\mathcal{J}}_{\tilde \Omega})$ is also a simple eigenvalue whenever $\tilde \Omega$ satisfies \eqref{eqdelta} substituting $\delta$ with $\delta_i$. Hence, as $s_{{\mathcal{J}}_{\Omega}}$ is a finite collection, the result follows if we take $\delta = \min \{\delta_1, ..., \delta_k \}$ setting $s_{{\mathcal{J}}_{\tilde \Omega}}$ in a natural form.   
\end{proof}

Now, we are ready to obtain the convergence of single eigenvalues given by a sequence of bounded open sets.
\begin{lemma} \label{cormu}
Let $\Omega_n \subset \R^N$ be a sequence of bounded open sets with 
$$
|\Omega \setminus \Omega_n| + |\Omega_n \setminus \Omega| \to 0, \quad \textrm{ as } n \to \infty
$$
for some bounded open set $\Omega \subset \R^N$. Then, if $\tilde \mu(\Omega)$ is an eigenvalue for $\tilde{\mathcal{J}}_{\Omega}$, there exists a family of eigenvalues $\tilde \mu(\Omega_n) \in \sigma(\tilde{\mathcal{J}}_{\Omega_n})$ such that  
$$
\tilde \mu(\Omega_n) \to \tilde \mu(\Omega), \quad \textrm{ as } n \to \infty.
$$
\end{lemma}
\begin{proof}
We just need to fix a small neighborhood for the single eigenvalue $\tilde \mu(\Omega)$ applying Lemma \ref{theofse} and Lemma \ref{lemJt}. 
\end{proof}

Next, let us proof Theorem \ref{singleeig}.
\begin{proof}[Proof of Theorem \ref{singleeig}]
It follows from Remarks \ref{EiPro} and \ref{EiPro2} that the eigenvalues of ${\mathcal{J}}_{\Omega_n}$ with finite multiplicity are all nonzero. 
Hence, the convergence of $\mu(\Omega_n)$ follows from Lemma \ref{lemig} and Lemma \ref{cormu}.  
Finally, the convergence of $\lambda(\Omega_n)$ is guaranteed by the  expression $\lambda(\Omega_n) = 1 - \mu_k(\Omega_n)$ obtained in Remark \ref{propeig}.
\end{proof}

Finally, let us consider two families of open sets discussing continuity of eigenvalues for the integral operators ${\mathcal{J}}_{\Omega}$ and ${\mathcal{B}}_{\Omega}$. First, we look at a family of open sets with rough boundary. Next, we analyse a periodically perforated domain. Below, we illustrate each family in Figure 1 and Figure 2 respectively. 

\begin{example}[Open sets with rough boundary] { \rm
Let us consider the following family of domains
$$
\Omega_n = \left\{ (x,y) \in \R^2 \, : \, x \in (0,1) \textrm{ and } 0 < y < 1 + \frac{\sin (2 \pi n x )}{n} \right\}.
$$
The family $\Omega_n$ can be seen as a perturbation of the unit square $\Omega = (0,1)^2$ and has been studied by  many authors; see e.g.,  \cite{arrieta_simone, Barbosa, CFP} and references therein. 

It is not difficult to see that 
$$
|\Omega \setminus \Omega_n| + |\Omega_n \setminus \Omega| = \frac{2 n \int_0^{1/2n} \sin(2 \pi n x) \, dx}{n} = \frac{2}{\pi n} \to 0 \quad \textrm{ as } n \to \infty.
$$
Consequently, we may apply Corollary \ref{singleeig}, Lemma \ref{theofse}, and Corollary \ref{corsimple} to this family of open sets evaluating the behavior of their eigenvalues.  }

\begin{figure}[htp] 
\centering \scalebox{0.4}{\includegraphics{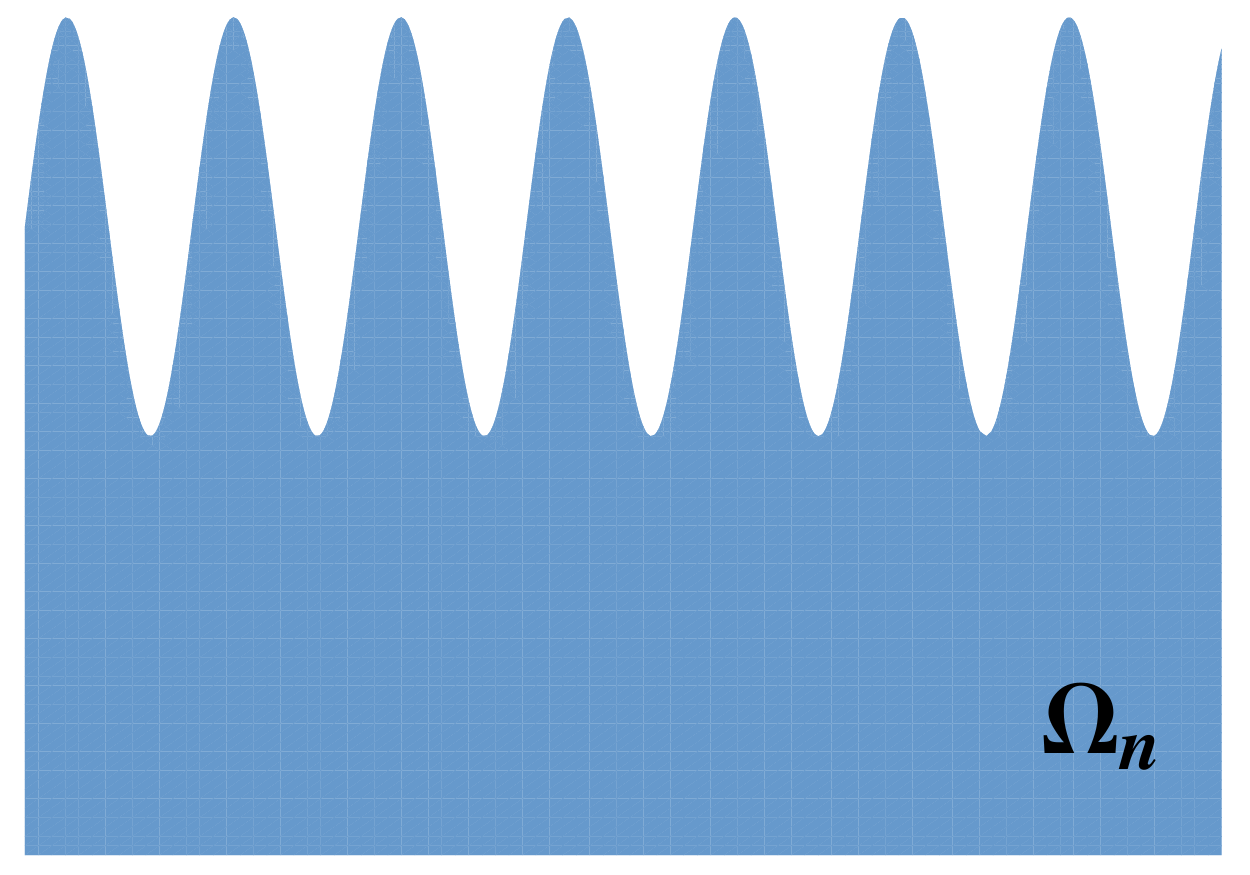}}
\caption{A family of open sets with rough boundary.}
\label{fig1} 
\end{figure}

\end{example}

\begin{example}[Perforated domains] { \rm 
Let $Q \subset \R^N$ be the following cell
$$
Q = (0,l_1) \times (0,l_2) \times ... \times (0,l_N).
$$ 
We perforate $\Omega \subset \R^N$ removing from it a set $A^\epsilon$ of periodically distributed holes set as follows:
Take any open set $A \subset Q$ such that $Q \setminus A$ is a measurable set with $|Q \setminus A| \neq 0$.
Denote by $\tau_\epsilon(A)$ all translated images of $\epsilon \bar A$ of the form 
$\epsilon(kl + A)$ for $k \in \mathbb{Z}^N$ and $kl=(k_1 l_1, ..., k_N l_N)$. 
Now define $A^\epsilon = \Omega \cap \tau_\epsilon(A)$
introducing our perforated domain as 
\begin{equation} \label{PerfDom}
\Omega^\epsilon = \Omega \setminus A^\epsilon, \quad \epsilon >0.
\end{equation}

Notice that, if the measure of the set $A$ is nonzero, then $|\Omega \setminus \Omega^\epsilon| + |\Omega^\epsilon \setminus \Omega|$ does not converge to zero as $\epsilon \to 0$. Thus, Theorem \ref{singleeig} and Lemma \ref{theofse}, as well Corollary \ref{corsimple}, can not be applied to this family of open sets.

Indeed, it follows from \cite[Lemma 3.1 and Section 4.1]{PR} that the first eigenvalue $\lambda_1(\Omega^\epsilon)$ of the nonlocal Dirichlet operator $\mathcal{B}_{\Omega^\epsilon}$ converges to a value $\beta_1$ as $\epsilon \to 0$ which satisfies $\beta_1 \in (0,1)$, and  
\begin{equation} \label{eq780}
- \frac{\beta_1}{\mathcal{X}} \phi^*(x) =  \mathcal{B}_{\Omega} \phi^*(x) - \frac{(1-\mathcal{X})}{\mathcal{X}} \phi^*(x), \quad x \in \Omega,
\end{equation}
for a strictly positive function $\phi^* \in L^2(\Omega)$, with $\phi^*(x) \equiv 0$ in $\R^N \setminus \Omega$, and a positive constant $\mathcal{X}$  
$$\mathcal{X} = \frac{|Q \setminus A|}{|Q|}$$
which is gotten by the limit of the characteristic function of the open sets $\Omega^\epsilon$ as $\epsilon \to 0$.

We have:
\begin{cor}
$\beta_1$ is the first eigenvalue of $\mathcal{B}_{\Omega}$, if and only if, $|A|=0$, that is, when $\Omega$ is weakly perforated.
\end{cor}
\begin{proof}
If $\beta_1$ is the first eigenvalue of $\mathcal{B}_{\Omega}$ and satisfies \eqref{eq780}, taking, $\phi^*$ as a test function in equation \eqref{eq780}, we get that 
$$
- \frac{(1-\mathcal{X}-\beta_1)}{\mathcal{X}} \| \phi^*\|_{L^2(\Omega)}^2 = \frac{1}{2} \int_{\R^N} \int_{\R^N} J(x-y) (\phi^*(y) - \phi^*(x))^2 dy dx  \geq \beta_1 \| \phi^*\|_{L^2(\Omega)}^2
$$
and then, $\beta_1 (1-\mathcal{X}) \geq (1-\mathcal{X})$. Since $\beta_1 \in (0,1)$, we obtain $\mathcal{X} =1$, which implies $|A|=0$.

Reciprocally, if $|A|=0$, then $|\Omega \setminus \Omega^\epsilon| + |\Omega^\epsilon \setminus \Omega|=0$ for all $\epsilon>0$, and then, we can apply Theorem \ref{singleeig} obtaining $\lambda_1(\Omega^\epsilon) \to \lambda_1(\Omega) = \beta_1$ as $\epsilon \to 0$, completing the proof.
\end{proof}
}

\begin{figure}[htp] 
\centering \scalebox{0.4}{\includegraphics{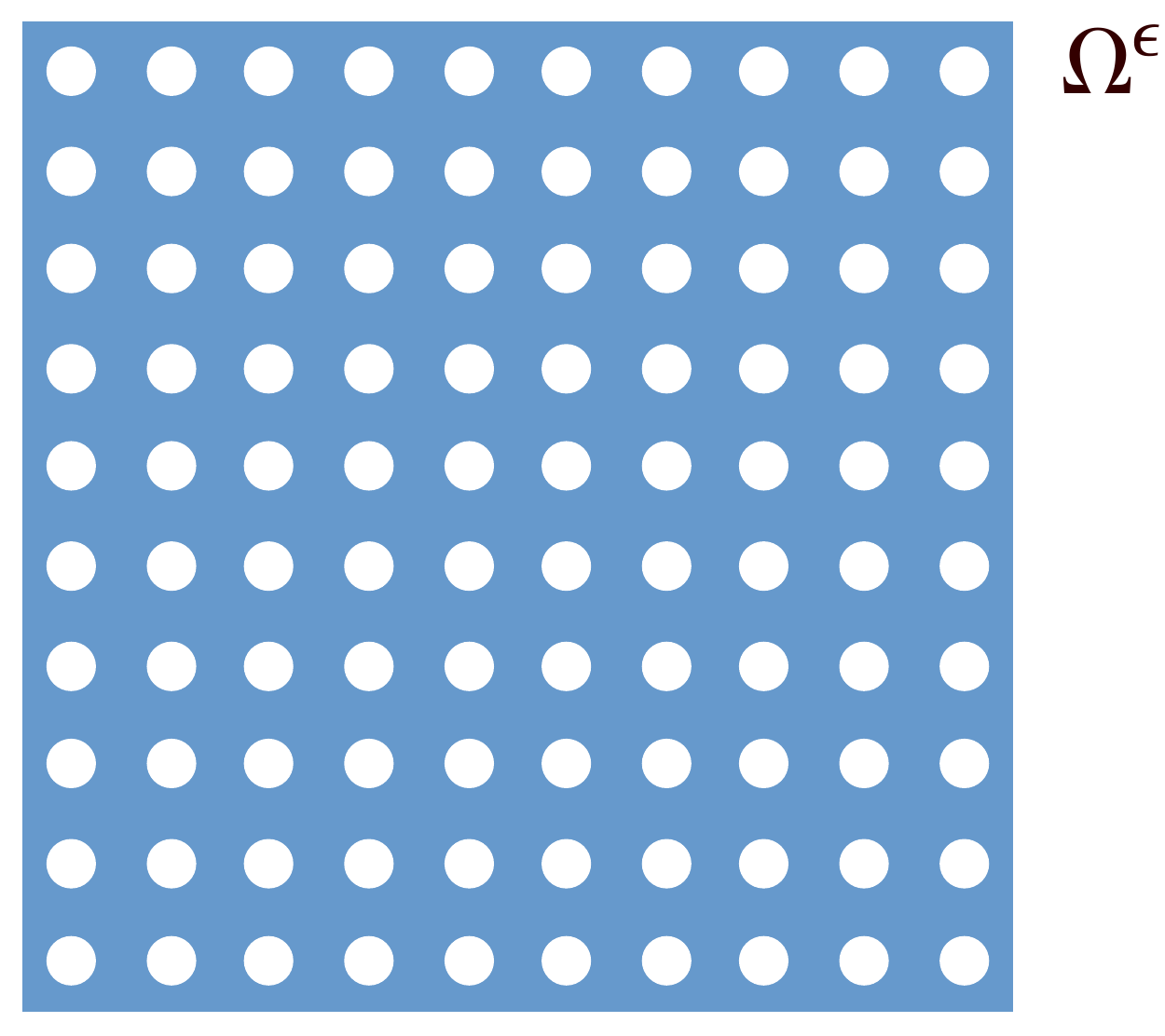}}
\caption{A periodic perforated domain $\Omega^\epsilon= (0,1)^2  \setminus A^\epsilon$.}
\label{fig1} 
\end{figure}

\end{example}

\section{Domain derivative of simple eigenvalues} \label{deriS}

In this section, we perturb simple eigenvalues of operators $\mathcal{J}_{\Omega}$ and $\mathcal{B}_{\Omega}$ getting derivatives with respect to the domain $\Omega$. We use the approach introduced in \cite{Henry} perturbing a fixed domain $\Omega$ by diffeomorphisms. As a consequence, we extend the expression obtained to the domain derivative for the first eigenvalue in \cite{GR} for any simple one in the spectral set of $\mathcal{J}_{\Omega}$ and $\mathcal{B}_{\Omega}$. 

Let $\Omega \subset \R^N$ be an open bounded set $\mathcal{C}^1$-regular.
If $h:\Omega \mapsto \R^N$ is a $\mathcal{C}^1$ imbedding, that is, a diffeomorphism to its image, we set the composition map $h^*$ (sometimes called pull-back) by 
$$
h^* v(x) = (v \circ h)(x), \quad x \in \Omega,
$$
when $v$ is any given function defined on $h(\Omega)$. It is not difficult to see $h^*: L^2(h(\Omega)) \mapsto L^2(\Omega)$ is an isomorphism with inverse 
$(h^*)^{-1} = (h^{-1})^*$. 

For such imbedding $h$ and a bounded region $\Omega$, one has 
\begin{equation} \label{Jh}
\left( \mathcal{J}_{h(\Omega)} v \right)(y) = \int_{h(\Omega)} J(y - w) v(w) dw, \quad \forall y \in h(\Omega),
\end{equation}
setting $\mathcal{J}_{h(\Omega)}: L^2(h(\Omega)) \mapsto L^2(h(\Omega))$.
On the other hand, we can use the pull-back operator $h^*$ to consider $h^* \mathcal{J}_{h(\Omega)} {h^*}^{-1}: L^2(\Omega) \mapsto L^2(\Omega)$ given by
\begin{equation} \label{hJh}
h^* \mathcal{J}_{h(\Omega)} {h^*}^{-1} u(x) = \int_{h(\Omega)} J(h(x) - w) (u \circ h^{-1})(w) dw, \quad \forall x \in \Omega.
\end{equation}

As we have already mentioned, expressions \eqref{Jh} and \eqref{hJh} are the customary way to describe motion or deformation of regions. 
\eqref{Jh} is called the Lagrangian description, and \eqref{hJh} the Eulerian one. 
The former is written in a fixed coordinate while the Lagrangian does not. 
It is easy to see 
\begin{equation} \label{eq820}
h^* \mathcal{J}_{h(\Omega)} {h^*}^{-1} u(x) = \int_{h(\Omega)} J(y - w) v(w) dw = \left( \mathcal{J}_{h(\Omega)} v \right)(y)
\end{equation}
if we take $y=h(x)$ and $v(y) = (u \circ h^{-1})(y) = {h^*}^{-1} u(y)$ for $y \in h(\Omega)$.

Notice $h^* \mathcal{J}_{h(\Omega)} {h^*}^{-1}$ is a compact operator since $h^*$ and ${h^*}^{-1}$ 
are isomorphisms and $\mathcal{J}_{h(\Omega)}$ is compact.
On the other side, if we maintain the Lebesgue measure, it is not a self-adjoint operator.

In fact, if we change  the $L^2(\Omega)$ measure using the determinant of the Jacobian matrix $Dh$ of $h$, 
we do obtain a self-adjoint operator. 
As $J$ is even,  by a change of variables, we have
\begin{eqnarray*}
\int_\Omega \varphi(x) h^* \mathcal{J}_{h(\Omega)} {h^*}^{-1} u(x) |{\rm det}(Dh(x))| dx =  
\int_{h(\Omega)} (\varphi \circ h^{-1})(y) \int_{h(\Omega)} J(y-w) (u \circ h^{-1})(w) dw dy   \\
 =  \int_\Omega \left(  \int_{h(\Omega)} J(h(z) - y) (\varphi \circ h^{-1})(y) dy \right) u(z) |{\rm det}(Dh(z))| dz \qquad \qquad  \\
 = \int_\Omega h^* \mathcal{J}_{h(\Omega)} {h^*}^{-1} \varphi(z) \; u(z) \; |{\rm det}(Dh(z))| dz. \qquad \qquad \qquad \qquad 
\end{eqnarray*}
Consequently, if we change the measure of $L^2(\Omega)$ taking  
$$
\hat L^2(\Omega) = \left\{  u:\Omega \mapsto \R \; : \; \int_\Omega u^2(x) |{\rm det}(Dh(x))| dx < \infty  \right\} 
$$
we have that $h^* \mathcal{J}_{h(\Omega)} {h^*}^{-1}: \hat L^2(\Omega) \mapsto \hat L^2(\Omega)$ is a compact self-adjoint operator in $\hat L^2(\Omega)$.
As $h$ is an imbedding, there exists $c>0$ such that $|{\rm det}(Dh)| \geq c>0$ in $\Omega$, and then, $\hat L^2(\Omega)$ is well defined. 
Thus, we can conclude that $\sigma \left( h^* \mathcal{J}_{h(\Omega)} {h^*}^{-1} \right) \subset \R$ for any imbedding $h:\Omega \mapsto \R$.

We have the following result:
\begin{prop} \label{peigenh}
Let $h:\Omega \mapsto \R$ be an imbedding. Then, $\mu \in \R$ is an eigenvalue of $h^* \mathcal{J}_{h(\Omega)} {h^*}^{-1}$, if and only if, is an eigenvalue for $\mathcal{J}_{h(\Omega)}$.
\end{prop}
\begin{proof}
Indeed, it follows from \eqref{eq820} that
$$
h^* \mathcal{J}_{h(\Omega)} {h^*}^{-1} u(x) = \mu u(x), \quad x \in \Omega, 
$$
if and only if, 
$$
\mathcal{J}_{h(\Omega)} v(y) = \mu v(y), \quad y \in h(\Omega),
$$
for $v(y) = (u \circ h^{-1})(y)$ with $y \in h(\Omega)$. Also, since ${h^*}^{-1}: L^2(\Omega) \mapsto L^2(h(\Omega))$ is an isomorphism, $u \neq 0$, if and only if, $v \neq 0$.
\end{proof}

Now, let us study differentiability properties of simple eigenvalues $\mu_{h(\Omega)}$ of $\mathcal{J}_{h(\Omega)}$ with respect to $h$. 
For this, we denote by $\rm{Diff}^1(\Omega) \subset \mathcal{C}^1(\Omega, \R^N)$ 
the set of $\mathcal{C}^1$-functions $h:\Omega \mapsto \R$ which are imbeddings considering the map 
\begin{eqnarray*}
F: \rm{Diff}^1(\Omega) \times \R \times L^2(\Omega) & \mapsto & L^2(\Omega) \times \R \\
(h,\mu, u) & \mapsto & \left(  \left(h^* \mathcal{J}_{h(\Omega)} {h^*}^{-1}  - \mu \right) u, \int_\Omega u^2(x)  |{\rm det}(Dh(x))| dx  \right).
\end{eqnarray*} 

It is not difficult to see that $\rm{Diff}^1(\Omega)$ is an open set of $\mathcal{C}^1(\Omega, \R^N)$ which denotes the space of $\mathcal{C}^1$-functions from $\Omega$ into $\R^N$ whose derivatives extend continuously to the closure $\bar \Omega$ with the usual supremum norm. 
Hence, $F$ can be seen as a map defined between Banach spaces.

Notice, if $\mu_0 \in \R$ is an eigenvalue for $\mathcal{J}_{\Omega}$ for some $u_0 \in L^2(\Omega)$ with $\int_\Omega u_0^2(x) dx = 1$, then $F(i_\Omega, \mu_0, u_0) = (0,1)$ where $i_\Omega \in \rm{Diff}^1(\Omega)$ denotes the inclusion map of $\Omega$ into $\R^N$.
On the other side, whenever $F(h, \mu, u) = (0,1)$, we have from Proposition \ref{peigenh} that 
$$
\mathcal{J}_{h(\Omega)} v(y) = \mu v(y), \quad y \in h(\Omega), 
\quad \textrm{ with } \quad 
\int_{h(\Omega)} v^2(y) dy =1
$$
where $v(y) = (u \circ h^{-1})(y)$ for $y \in h(\Omega)$.
In this way, we can use the map $F$ to deal with eigenvalues and eigenfunctions of $\mathcal{J}_{h(\Omega)}$ and $h^* \mathcal{J}_{h(\Omega)} {h^*}^{-1}$ perturbing the eigenvalue problem to the fixed domain $\Omega$ by diffeomorphisms $h$.

\begin{lemma} \label{deri}
Let $\mu_0$ be a simple eigenvalue for $\mathcal{J}_{\Omega}$ with corresponding normalized eigenfuction $u_0$ and 
$J \in \mathcal{C}^1(\R^N, \R)$ satisfying $(\rm{H})$. 
Then, there exists a neighbourhood $\mathcal{V}$ of inclusion $i_{\Omega} \in \mathop{\rm Diff}^1(\Omega)$, and $\mathcal{C}^1$-functions 
$u_h$ and $\mu_h$ from $\mathcal{V}$ into $L^2(\Omega)$ and $\mathbb{R}$ respectively satisfying 
$$
h^* \mathcal{J}_{h(\Omega)} {h^*}^{-1} u_h(x) = \mu_h u(x), \quad x \in \Omega, 
$$
with $u_h \in \mathcal{C}^1(\Omega)$ for all $h \in \mathcal{V}$.  

Moreover, $\mu_h$ is a simple eigenvalue with $(\mu_{i_{\Omega}}, u_{i_\Omega} ) = ( \mu_0, u_0)$ and domain derivative 
\begin{equation} \label{dode}
\frac{\partial \mu}{\partial h}(i_\Omega) \cdot V = \mu_0 \int_{\partial \Omega} u_0^2 \; V \cdot N_\Omega dS \quad 
\forall V \in \mathcal{C}^1(\Omega, \R^N).
\end{equation}
\end{lemma}
\begin{proof}
Under the additional condition $J \in \mathcal{C}^1(\R^N, \R)$, we get from \cite{Die} that the map $F$ is a $\mathcal{C}^1$-function between Banach spaces (see also \cite[Chapter 2]{Henry}).
In fact, $F$ is linear with respect to the variables $\mu \in \R$ and $u \in L^2(\Omega)$. Also, it is of class $\mathcal{C}^1$ with respect to $h$, since expressions
\begin{equation} \label{eq1041}
\begin{gathered}
h^* \mathcal{J}_{h(\Omega)} {h^*}^{-1} u(x)  =  \int_{h(\Omega)} J(h(x) - w) (u \circ h^{-1})(w) dw \\
\qquad \qquad \qquad =  \int_\Omega J(h(x) - h(z)) u(z)  |{\rm det}(Dh(z))| dz, \quad x \in \Omega,
\end{gathered}
\end{equation}
and $\int_\Omega u^2(x)  |{\rm det}(Dh(x))| dx$ are set by compositions among smooth functions $J$, ${\rm det}$ and $h$ 
which define $\mathcal{C}^1$-maps in the variable $h \in \rm{Diff}^1(\Omega)$.

Next, since $\mu_0$ is a simple eigenvalue with $F(i_\Omega, \mu_0, u_0) = (0,1)$, 
we are in condition to apply Implicit Function Theorem to $F$ at $(i_\Omega, \mu_0, u_0) \in \rm{Diff}^1(\Omega) \times \R \times L^2(\Omega)$.
First, we see 
\begin{eqnarray*}
\frac{\partial F}{\partial(\mu, u)}(i_\Omega, \mu_0, u_0): \R \times L^2(\Omega) & \mapsto & L^2(\Omega) \times \R \\
(\dot \mu, \dot u) & \mapsto & \left(   ( \mathcal{J}_{\Omega} - \mu_0) \dot u + \dot \mu u_0, 2 \int_\Omega u_0 \, \dot u \, dx  \right)
\end{eqnarray*}
is an isomorphism. In fact, since $\mu_0$ is a simple eigenvalue, its eigenfunction $u_0$ is orthogonal to the image of the operator $(\mathcal{J}_{\Omega} - \mu_0)$ satisfying 
$
L^2(\Omega) = {\rm R}(\mathcal{J}_{\Omega} - \mu_0) \oplus [u_0].
$

Thus, for any $f \in L^2(\Omega)$, there exists a  unique $w \in {\rm R}(\mathcal{J}_{h(\Omega)} - \mu_0)$ such that 
$$
(\mathcal{J}_{\Omega} - \mu_0) w = f - \dot \mu u_0 \quad \textrm{ with } \quad \dot \mu = \int_\Omega f u_0
$$
since for such $\dot \mu$, $f - \dot \mu u_0$ is orthogonal to $u_0$ in $L^2(\Omega)$ belonging to ${\rm R}(\mathcal{J}_{\Omega} - \mu_0)$.
Consequently, for all $(f,a) \in L^2(\Omega) \times \R$, we can take unique $\dot u = w + \frac{a}{2} u_0$ and $\dot \mu = \int_\Omega f u_0$ such that
$$
\frac{\partial F}{\partial(\mu, u)}(i_\Omega, \mu_0, u_0)(\dot \mu, \dot u) = (f,a).
$$

Therefore, by the Implicit Function Theorem, there exist $\mathcal{C}^1$-functions $h \mapsto (\mu_h, u_h)$ such that 
$
F(h, \mu_h, u_h) = (0,1)
$
whenever $\| h - i_\Omega \|_{\mathcal{C}^1(\Omega,\R^N)}$ is sufficiently small.
Thus, we have a family of simple eigenvalues $\mu_h$ and corresponding eigenfunctions $v_h=(u_h \circ h^{-1})$ for $\mathcal{J}_{h(\Omega)}$ defined by any $h$ in a neighborhood of $i_\Omega \in \rm{Diff}^1(\Omega)$ which is still differentiable with respect to $h$. 

Finally, we compute the derivative of $\mu_h$ at $h=i_\Omega$.
For this, it is enough to consider a curve of imbeddings $h(t,x) = x + t V(x)$ for a fixed $V \in \mathcal{C}^1(\Omega,\R^N)$ taking the Gateaux derivative at $t=0$.

Notice that 
$$
h(t)^* \mathcal{J}_{h(t,\Omega)} {h(t)^*}^{-1} u_{h(t)}(x) = \mu_{h(t)} u_{h(t)}, \quad x \in \Omega,
$$
and then, 
\begin{equation} \label{eq960}
\frac{\partial}{\partial t} \left( h(t)^* \mathcal{J}_{h(t,\Omega)} {h(t)^*}^{-1} u_{h(t)}(x) \right) \Big|_{t=0} 
 = \frac{\partial \mu_{i_\Omega}}{\partial t} u_0  + \mu_0 \frac{\partial u_{i_\Omega}}{\partial t} \quad \textrm{ in } \Omega.
\end{equation}
Thus, in order to complete our proof, we need to compute the derivative of the left-side of \eqref{eq960}. 
We proceed as in \cite{Henry} using the anti-convective derivative $D_t$ in the reference region $\Omega$
$$
D_t = \frac{\partial}{\partial t} - U(t,x) \cdot \frac{\partial}{\partial x} \quad 
\textrm{ with } \quad U = {\frac{\partial h}{\partial x}}^{-1} \frac{\partial h}{\partial t}.
$$

By \cite[Lemma 2.1]{Henry}, we have 
\begin{equation} \label{eq1100}
D_t  \left( h(t)^* \mathcal{J}_{h(t,\Omega)} {h(t)^*}^{-1} u_{h(t)} \right) 
= h(t)^*\frac{\partial}{\partial t} \left(\mathcal{J}_{h(t,\Omega)} {h(t)^*}^{-1} u_{h(t)} \right)
\quad \textrm{ in } \Omega.
\end{equation}
Now, set $v(t,y) = {h(t)^*}^{-1} u_{h(t)}(y) = u_{h(t)}(h^{-1}(t,y))$, $y \in h(t,\Omega)$.
Then, from \eqref{eq820}, we get  
\begin{eqnarray*}
\frac{\partial}{\partial t} \left(\mathcal{J}_{h(t,\Omega)} {h(t)^*}^{-1} u_{h(t)} \right) \Big|_{t=0} 
& = & \frac{\partial}{\partial t} \left( \mathcal{J}_{h(t,\Omega)} v \right) \Big|_{t=0} \\
& = & \frac{\partial}{\partial t} \left( \int_{h(t,\Omega)} J(y - w) v(t, w) dw \right) \Big|_{t=0} \quad \textrm{ for } y \in h(t,\Omega).
\end{eqnarray*}

Due to \cite[Theorem 1.11]{Henry}, we can compute domain derivatives for integrals obtaining 
\begin{eqnarray*}
& & \frac{\partial}{\partial t} \left(\mathcal{J}_{h(t,\Omega)} {h(t)^*}^{-1} u_{h(t)} \right) \Big|_{t=0} 
 =  \frac{\partial}{\partial t} \left( \mathcal{J}_{h(t,\Omega)} v \right) \Big|_{t=0} \\ 
& & \qquad \qquad =  \int_\Omega J(x-w) (D_t u)(0,w) \, dw + \int_{\partial \Omega} J(x-z) u_0(z) \, (V \cdot N_\Omega)(z) \, dS(z)
\end{eqnarray*}
where $N_\Omega$ is the unitary normal vector to $\partial \Omega$. 

Notice that the last integral on $\partial \Omega$ is well defined. 
Since $J$ is $\mathcal{C}^1$, the eigenfunctions $u_h$ and their derivatives can be continuously extended to the border $\partial \Omega$.
Thus, $u_h \in \mathcal{C}^1(\Omega)$, and we can take the trace of $u_h$ on $\partial \Omega$.  

Consequently, from \eqref{eq960} and \eqref{eq1100},   we get 
\begin{eqnarray*}
&& \frac{\partial \mu_{i_\Omega}}{\partial t} u_0  + \mu_0 \frac{\partial u_{i_\Omega}}{\partial t} 
 = \left[ U(t,x) \cdot \frac{\partial}{\partial x} \left( h(t)^* \mathcal{J}_{h(t,\Omega)} {h(t)^*}^{-1} u_{h(t)} \right)  +    \frac{\partial}{\partial t} \left(\mathcal{J}_{h(t,\Omega)} {h(t)^*}^{-1} u_{h(t)} \right) \right]_{t=0}   \\
& & \qquad \qquad =  V \cdot \frac{\partial}{\partial x} \left( \mathcal{J}_{\Omega} u_0 \right) + \mathcal{J}_{\Omega} D_t u 
 + \int_{\partial \Omega} J(\cdot -z) u_0(z) \, (V \cdot N_\Omega)(z) \, dS(z) \quad \textrm{ in } \Omega.
\end{eqnarray*}   
Hence, multiplying by $u_0$ and integrating on $\Omega$, we obtain 
\begin{eqnarray*}
& & \frac{\partial \mu_{i_\Omega}}{\partial t} + \int_\Omega \mu_0 u_0 \frac{\partial u_{i_\Omega}}{\partial t} dx 
= \int_\Omega \mu_0 u_0 \left( V \cdot \nabla u_0 + \frac{\partial u_{i_\Omega}}{\partial t} - V \cdot \nabla u_0 \right) dx \\
& & \qquad \qquad + \int_{\partial \Omega} \left( \int_\Omega J(x -z) u_0(x) \, dx \right) u_0(z) \, (V \cdot N_\Omega)(z) \, dS(z)
\end{eqnarray*}
which implies 
$$
\frac{\partial \mu_{i_\Omega}}{\partial t} = \mu_0 \int_{\partial \Omega} u_0^2(z) \, (V \cdot N_\Omega)(z) \, dS(z)
$$
completing the proof.
\end{proof}

Therefore, as a direct consequence of Lemma \ref{deri} and items $a)$ and $c)$ from Remark \ref{propeig}, 
we get Theorem \ref{cderi} concerning  the Dirichlet problem \eqref{eigeneq}.

\begin{remark}
From Corollary \ref{corl1}, we know $\lambda_1(\Omega^*)$ is simple, and a critical point to the map 
$$h \in {\rm Diff}^1(\Omega^*) \mapsto \Big( \lambda_1(h(\Omega^*)), |h(\Omega^*)|=|\Omega^*| \Big).$$ 
Hence, from Theorem \ref{cderi} 
$$
0=\int_{\partial \Omega^*} u_1^2 \; V \cdot N_{\Omega^*} dS 
\quad \textrm{ for all } V \in \mathcal{C}^1(\Omega^*, \R^N) 
\textrm{ such that } \int_{\partial \Omega^*} V \cdot N_{\Omega^*} = 0.
$$
Therefore, the first eigenfunction $u_1$ associated to $\lambda_1(\Omega^*)$ satisfies the boundary condition 
$u_1(x) = c$ on $\partial \Omega$ for some constant $c \geq 0$.
\end{remark}

\begin{remark}
Finally, let us give an example which shows that in general, the first eigenvalue $\lambda_1(\Omega)$ of \eqref{eigeneq} does not possess a maximizer among open bounded sets with $|\Omega|=constant$. 

For this, let $h: (0,1)^2 \mapsto (0,a) \times (0,1/a) \subset \mathbb{R}^2$ be the imbedding $h(x_1, x_2) = ( a x_1, (1/a) \,  x_2)$ for any $a>0$.
Notice that ${\rm det}(Dh) =1$ and $|h((0,1)^2)|=1$ for all $a$. 
Also, from \eqref{eq1041} we have 
$$
h^* \mathcal{J}_{h((0,1)^2)} {h^*}^{-1} u(x)  =  \int_{(0,1)^2} J(a(x_1-y_1), (1/a)(x_2 - y_2)) u(y) dy, \quad \forall x \in (0,1)^2.
$$
Hence, since $J(x) \to 0$ as $|x| \to +\infty$ by hypothesis ${\rm (H)}$, we obtain that $h^* \mathcal{J}_{h((0,1)^2)} {h^*}^{-1} u(x) \to 0$ as $a \to 0$, for all $x \in (0,1)^2$ and $u \in L^2(\Omega)$.
Therefore, one can get from Proposition \ref{peigenh} and Remark \ref{mu1} that $\mu_1(h((0,1)^2)) \to 0$ as $a \to 0$ implying that $\lambda_1(h((0,1)^2)) \to 1$ as $a \to 0$. As $1 \in \sigma_{ess} (\mathcal{B}_\Omega)$ for any open set $\Omega$, we conclude our assertion. 
\end{remark}

\vspace{0.5 cm}

{\bf Acknowledgements.} 
The first author (RB) has been supported by Fondecyt (Chile)
Project  \# 116--0856,   and the second one (MCP) by CNPq 303253/2017-7 and FAPESP 2019/06221-5 (Brazil). Finally, we would like to mention that this work was done while the second author was visiting the Instituto de F\'isica at P. Universidad Cat\'olica de Chile. He kindly expresses his gratitude to the institute.

\typeout{get arXiv to do 4 passes: Label(s) may have changed. Rerun}

\end{document}